\documentclass[a4paper,12pt]{article}
\usepackage[top=3cm,bottom=3cm,left=2.5cm,right=2.5cm]{geometry}
\usepackage{t1enc}
\usepackage[latin1]{inputenc}
\usepackage[english]{babel}
\usepackage{amsmath,amsthm}
\usepackage{amsfonts}
\usepackage{latexsym}
\usepackage[dvips]{graphicx}
\usepackage{graphicx}
\usepackage{color}
\usepackage{amsmath}
\usepackage{amssymb}
\usepackage{amsbsy}
\usepackage{amsfonts,mathrsfs}
\usepackage{indentfirst}
\usepackage{tikz}
\usepackage{booktabs}
\usepackage{diagbox}
\usepackage{multirow}
\usepackage{makecell}
\usepackage{longtable}
\usepackage{subfigure}
\usepackage{rotating}

\theoremstyle{definition}
\newtheorem{Def}{Definition}[section]

\theoremstyle{plain}
\newtheorem{Thm}[Def]{Theorem}
\newtheorem*{Dumas}{Dumas' Theorem}
\newtheorem{Lem}[Def]{Lemma}
\newtheorem{Prop}[Def]{Proposition}
\newtheorem{Cor}[Def]{Corollary}

\newtheorem{Conj}[Def]{Conjecture}

\numberwithin{equation}{section}

\newcommand{\Mod}{\text{\rm mod~}}

\title{
	A note on an irreducible class of polynomials over integers\thanks{This work is supported by the National Natural Science Foundation of China(Grant No. 12171163).}
}
\author{
	Weilin Zhang, Pingzhi Yuan\thanks{Corresponding author: yuanpz@scnu.edu.cn}\\
	{\small School of Mathematical Sciences, South China Normal University,}\\
	{\small Guangzhou, 510631, P. R. China}
}
\date{}

\begin{document}

\maketitle

\addtolength{\baselineskip}{+0.6mm}
{\bf Abstract}: In this note, we prove an irreducibility criterion for the polynomial of the form $f(x) = a_{n}x^{n} + a_{n-1}x^{n-1} + \cdots + a_{m}x^{m} + p^{u} \in \mathbb{Z}[x]$, where $p$ is a prime number, $u \geqslant 1$, $\gcd(u, m) = 1$, $p \nmid a_{m}$ and $p^{u} > |a_{n}| + |a_{n-1}| + \cdots + |a_{m}|$. In particular, we show that the conjecture of Koley and Reddy is true.



\baselineskip=0.30in
\section{Introduction}

Many researchers studied the irreducibility of polynomials in $\mathbb{Z}[x]$ having a constant term as a prime number or prime powers. For instance, Weisner \cite{W1934} proved that if $p$ is a prime number and $n \geqslant 2$, $u \geqslant 1$, then $x^{n}\pm x \pm p^{u}$ is irreducible whenever $p^{u} > 2$. Jonassen \cite{J1967} proved that $x^{n}\pm x^{m} \pm 4$ is irreducible except for six distinct families of polynomails. Recently, Koley and Reddy \cite{KR2023} proved that apart from cyclotomic factors, $x^{n}\pm x \pm 2$ has exactly one non-reciprocal irreducible factor.

For a more general case on this problem, Panitopol and Stef\"{a}nescu \cite{PS1985} proved the following Proposition \ref{Prop1}.

\begin{Prop} \label{Prop1}
	Let $f(x) = a_{n}x^{n} + a_{n-1}x^{n-1} + \cdots + a_{1}x + p \in \mathbb{Z}[x]$, where $p$ is a prime number. If  $p > |a_{n}| + |a_{n-1}| + \cdots + |a_{1}|$, then $f(x)$ is irreducible over $\mathbb{Q}$.
\end{Prop}

Note that $f(x)$ and $-f(x)$ have the same irreducibility. For the sake of simplicity of presentation, we always suppose that $f(x)$ has a positive constant term. A. I. Bonciocat and N. C. Bonciocat extended Proposition \ref{Prop1} to prime powers in \cite{BB2009} and \cite{BB2006}. In \cite{BB2009}, they proved that if $p\nmid a_{0}a_{1}$, $p^{u} > |a_{1}|p^{2e} + \sum_{i=2}^{n}|a_{0}^{i-1}a_{i}|p^{ie}$, $u \geqslant 1$, $e \geqslant 0$, then $a_{n}x^{n} + \cdots + a_{2}x^{2} + a_{1}p^{e}x + a_{0}p^{u}$ is irreducible over $\mathbb{Q}$. Setting $e = 0$ and $a_{0} = 1$, we get the following Proposition \ref{Prop2}.

\begin{Prop} \label{Prop2}
	Let $f(x) = a_{n}x^{n} + a_{n-1}x^{n-1} + \cdots + a_{1}x + p^{u} \in \mathbb{Z}[x]$, where $p$ is a prime number and $u \geqslant 1$. If $p \nmid a_{1}$ and $p^{u} > |a_{n}| + |a_{n-1}| + \cdots + |a_{1}|$, then $f(x)$ is irreducible over $\mathbb{Q}$.
\end{Prop}

It is worth mentioning that the condition $p\nmid a_{1}$ in Proposition \ref{Prop2} cannot be removed. For example, $x^{3} - x^{2} - 10x + 16 = (x-2)(x^{2} + x -8)$.

In \cite{BB2006}, they proved that if $p\nmid a_{0}a_{2}$, $p^{u} > |a_{0}a_{2}|p^{3e} + \sum_{i=3}^{n}|a_{0}^{i-1}a_{i}|p^{ie}$ and $u \not\equiv e ~(\Mod 2)$, $u \geqslant 1$, $e \geqslant 0$, then $a_{n}x^{n} + \cdots + a_{3}x^{3} + a_{2}p^{e}x^{2} + a_{0}p^{u}$ is irreducible over $\mathbb{Q}$. Setting $e = 0$ and $a_{0} = 1$, we get the following Proposition \ref{Prop3}.

\begin{Prop} \label{Prop3}
	Let $f(x) = a_{n}x^{n} + a_{n-1}x^{n-1} + \cdots + a_{2}x^{2} + p^{u} \in \mathbb{Z}[x]$, where $p$ is a prime number and $u \geqslant 1$. If $p \nmid a_{2}$, $2 \nmid u$ and $p^{u} > |a_{n}| + |a_{n-1}| + \cdots + |a_{1}|$, then $f(x)$ is irreducible over $\mathbb{Q}$.
\end{Prop}

As with Proposition \ref{Prop2}, the conditions $p\nmid a_{2}$ and $2\nmid u$ in Proposition \ref{Prop3} also cannot be removed. For example,
\begin{equation*}
	2x^{3} - 3x^{2} - 27 = (x-3)(2x^{2} + 3x + 9),
\end{equation*}
\begin{equation*}
	x^{4} + (2^{k+1}-1)x^{2} + 2^{2k} = (x^{2} + x + 2^{k})(x^{2} - x + 2^{k}), \text{~where~} k\geqslant 1.
\end{equation*}

Recently, Koley and Reddy tried to find the similar irreducibility criterion for the polynomial of the form $f(x) = a_{n}x^{n} + a_{n-1}x^{n-1} + \cdots + a_{3}x^{3} + p^{u} \in \mathbb{Z}[x]$, where $p$ is a prime number and $u \geqslant 1$. In \cite{KR2022}, they cleverly used the technique of reciprocal polynomials to porve the following Proposition \ref{Prop4}.

\begin{Prop} \label{Prop4}
	Let $f(x) = a_{n}x^{n} + a_{n-1}x^{n-1} + \cdots + a_{3}x^{3} + p^{u} \in \mathbb{Z}[x]$, where $p$ is a prime number and $u \geqslant 1$. If $p\nmid a_{n}a_{3}$, $3 \nmid u$ and $p^{u} > |a_{n}| + |a_{n-1}| + \cdots + |a_{3}|$, then $f(x)$ is irreducible over $\mathbb{Q}$.
\end{Prop}

From the proof of Proposition \ref{Prop4} in \cite{KR2022}, the condition $p\nmid a_{n}a_{3}$ cannot be replaced with $p\nmid a_{3}$. In this note, we will show that this replacement is feasible.

According to Propositions \ref{Prop2}--\ref{Prop4} and some examples in \cite{KR2022}, the authors proposed the following Conjecture \ref{Conj}.

\begin{Conj} \label{Conj}
	Let $f(x) = a_{n}x^{n} + a_{n-1}x^{n-1} + \cdots + a_{q}x^{q} + p^{u} \in \mathbb{Z}[x]$, where $p$ and $q$ are two prime numbers and $u \geqslant 1$. If $p\nmid a_{n}a_{q}$, $q \nmid u$ and $p^{u} > |a_{n}| + |a_{n-1}| + \cdots + |a_{q}|$, then $f(x)$ is irreducible over $\mathbb{Q}$.
\end{Conj}

In this note, we will show that Conjecture \ref{Conj} is true and the condition $p\nmid a_{n}a_{q}$ can be replaced with $p \nmid a_{q}$. For this purpose, we prove the following Theorem \ref{Thm}.

\begin{Thm} \label{Thm}
	Let $f(x) = a_{n}x^{n} + a_{n-1}x^{n-1} + \cdots + a_{m}x^{m} + p^{u} \in \mathbb{Z}[x]$, where $p$ is a prime number and $u \geqslant 1$. If $p \nmid a_{m}$, $\gcd(u, m) = 1$ and $p^{u} > |a_{n}| + |a_{n-1}| + \cdots + |a_{m}|$, then $f(x)$ is irreducible over $\mathbb{Q}$.
\end{Thm}

Clearly, let $m = q$ be a prime number. Then $\gcd(q, m) = 1$ implies $q \nmid  u$. Thus, we get the following Corollary \ref{Cor}, which claim that Conjecture \ref{Conj} is true.

\begin{Cor} \label{Cor}
	Let $f(x) = a_{n}x^{n} + a_{n-1}x^{n-1} + \cdots + a_{q}x^{q} + p^{u} \in \mathbb{Z}[x]$, where $p$ and $q$ are two prime numbers and $u \geqslant 1$. If $p\nmid a_{q}$, $q \nmid u$ and $p^{u} > |a_{n}| + |a_{n-1}| + \cdots + |a_{q}|$, then $f(x)$ is irreducible over $\mathbb{Q}$.
\end{Cor}

 The proof of Theorem \ref{Thm} is presented in Section \ref{Sec3}, which relies on Newton polygon method. In Section \ref{Sec2}, we briefly recall a useful result on Newton polygon without more proof details.

\section{Newton polygon method} \label{Sec2}

Throughout this note, the notations and terminologies about Newton polygon method freely used but not defined here, we refer the reader to \cite{B2003, B2014, BBCM2013}. The following celebrated result of Dumas \cite{D1906}, which palys a central role in Newton polygon method.

\begin{Dumas}
	Let $p$ be a prime number and let $f(x) = g(x)h(x)$, where $f(x), g(x)$ and $h(x)$ are polynomials in $\mathbb{Z}[x]$. Then the system of vectors of the segments for the Newton polygon of $f(x)$ with respect to $p$ is the union of the systems of vectors of the segments for the Newton polygons of $g(x)$ and $h(x)$ with respect to $p$.
\end{Dumas}

 For applications of Newton polygon method in the study of the irreducibility for various classes of polynomials, like for instance Bessel polynomials and Laguerre polynomials, we refer the reader to the work of Filaseta \cite{F1995}, Filaseta, Finch and Leidy \cite{FFL2005}, Filaseta, Kidd and Trifonov \cite{FKT2012}, Filaseta and Lam \cite{FL2002}, Filaseta and Trifonov \cite{FT2002}, Filaseta and Williams \cite{FWJ2003}, Hajir \cite{H1995}, and Sell \cite{S2004}.

\section{Proof of Theorem \ref{Thm}} \label{Sec3}

	In this section, we start by proving the Lemma \ref{Lem}, which is useful for our proof.

\begin{Lem}[Remark 4, \cite{KR2022}] \label{Lem}
	Let $f(x) = a_{n}x^{n} + a_{n-1}x^{n-1} + \cdots + a_{1}x + a_{0} \in \mathbb{C}[x]$ be any polynomial of degree $n$ and $|a_{0}| > |a_{n}| + |a_{n-1}| + \cdots + |a_{1}|$. Then every root of $f(x)$ lies outside the unit circle.
\end{Lem}

\begin{proof}
	Let $z$ be a root of $f(x)$ with $|z| \leqslant 1$. Then $f(z) = 0$, that is,
	\begin{equation} \label{F1}
		-a_{0} = a_{n}z^{n} + a_{n-1}z^{n-1} + \cdots + a_{1}z.
	\end{equation}
	Taking modulus on both sides of \eqref{F1} and applying the absolute value inequality, we have
	\begin{equation*}
		\begin{aligned}
			|a_{0}|
			& = |a_{n}z^{n} + a_{n-1}z^{n-1} + \cdots + a_{1}z| \\
			& \leqslant |a_{n}||z|^{n} + |a_{n-1}||z|^{n-1} + \cdots + |a_{1}||z| \\
			& \leqslant |a_{n}| + |a_{n-1}| + \cdots + |a_{1}|,
		\end{aligned}
	\end{equation*}
	which contradicts the hypothesis.
	
	Therefore, all the roots of $f(x)$ lies outside the unit circle.
\end{proof}

 The following Lemma \ref{lem} is another key point  for our proof. More details on Lemma \ref{lem}, we refer the reader to see the proof of Lemma 1.4 in \cite{BBCM2013}.

\begin{Lem} \label{lem}
	Let $A(x_{1}, y_{1})$ and $B(x_{2}, y_{2})$ be two integral points in the plane. Then if $\gcd(|x_{1}-x_{2}|, |y_{1}-y_{2}|)=1$, then there no other integral points on the segment $AB$ expect $A$ and $B$.
\end{Lem}

Now, we present the proof of our main result.

\begin{proof}[Proof of Theorem \ref{Thm}]
	Suppose that $f(x) = f_{1}(x)f_{2}(x)$ is a non-trival factorization of $f(x)$. We say
	\begin{equation*}
		f_{1}(x) = b_{s}x^{s} + b_{s-1}x^{s-1} + \cdots + b_{0} \in \mathbb{Z}[x],
	\end{equation*}
	\begin{equation*}
		f_{2}(x) = c_{t}x^{t} + c_{t-1}x^{t-1} + \cdots + c_{0} \in \mathbb{Z}[x],
	\end{equation*}
where $b_{s}c_{t}\neq 0$, $s \geqslant 1$, $t\geqslant 1$ and $s + t = n$.

Since $f(x) = f_{1}(x)f_{2}(x)$, comparing the constant term on both sides, we get
	\begin{equation*}
		p^{u} = b_{0}c_{0}.
	\end{equation*}
Without loss of generality, we may suppose that $|b_{0}| = p^{\alpha}$, $|c_{0}| = p^{u-\alpha}$, where $0 \leqslant \alpha \leqslant \dfrac{u}{2}$. Now we claim that $\alpha = 0$. Assume that $0 < \alpha \leqslant \dfrac{u}{2}$. Then $p \mid b_{0}$ and $p \mid c_{0}$. Since $p \nmid a_{m}$, there exsit $b_{i_{0}}$ and $c_{j_{0}}$ such that $p \nmid b_{i_{0}}c_{j_{0}}$, where $1 \leqslant i_{0} \leqslant s$ and $1\leqslant j_{0} \leqslant t$. Thus, the Newton polygon of $f_{1}(x)$ with respect to $p$ contains at least one segment with a negative slope. Similarly, the same is true for $f_{2}(x)$. However, since $p \nmid a_{m}$, the first two points in the Newton polygon of $f(x)$ with respect to $p$ are $A(0, u)$ and $B(m, 0)$. By $\gcd(u, m) = 1$ and Lemma \ref{lem}, there no other integral points on the segment $AB$ expect $A$ and $B$. Thus, the Newton polygon of $f(x)$ with respect to $p$ exactly contains one segment with a negative slope. This contradicts with Dumas' Theorem. Therefore, we get $\alpha = 0$ and $|b_{0}| = 1$.

Suppose that $z_{1}, z_{2}, \cdots, z_{s} \in \mathbb{C}$ are all the roots of $f_{1}(x)$, which are also the roots of $f(x)$. Applying the Vieta Theorem on $f_{1}(x)$, we have
	\begin{equation*}
		\prod_{i=1}^{t}|z_{i}| = \dfrac{|b_{0}|}{|b_{t}|} \leqslant 1,
	\end{equation*}
which means that $f(x)$ has a root within the unit closed circle.

On the other hand, since $p^{u} > |a_{n}| + |a_{n-1}| + \cdots + |a_{m}|$, by Lemma \ref{Lem}, all the roots of $f(x)$ lies outside the unit circle. This is a contradiction.

Therefore, $f(x)$ is irreducible over $\mathbb{Q}$.
\end{proof}

\end{document}